\documentclass[11pt]{article}
\usepackage{latexsym, amsmath, amssymb}

\usepackage{amsmath}
\usepackage{amsfonts}
\usepackage{amssymb}
\usepackage{amscd}
\usepackage{amsthm}
\usepackage{color}
\usepackage{verbatim}
\theoremstyle{plain}
\newtheorem*{theorem-nonumber}{Theorem}
\newtheorem{theorem}{Theorem}[section]
\newtheorem{proposition}[theorem]{Proposition}

\newtheorem{lemma}[theorem]{Lemma}
\newtheorem{definition}[theorem]{Definition}
\newtheorem{corollary}[theorem]{Corollary}
\newtheorem{example}[theorem]{Example}

\DeclareMathOperator{\Ext}{Ext} \DeclareMathOperator{\Hom}{Hom} \DeclareMathOperator{\Aut}{Aut}
  \DeclareMathOperator{\coker}{coker}

\def\Q{\mathbb Q}

\newcommand\sO{{\mathcal O}}

\newcommand\sQ{{\mathcal Q}}

\def\pee#1{\hbox{$ {\mathbf P}^{#1}$}}
  \def\peen{\hbox{$ {\mathbf  P}^n$}}

  \def \tab#1{\kern #1 truein}

  \def\OP#1{\hbox{${\cal O}_{{\mathbf P}^{#1}}$}}

  \def\Q{\hbox{${\cal Q}$}}

\begin{document}
  \title{Horrocks Correspondence on a Quadric Surface}
\author{F. Malaspina\thanks{Partially supported by GNSAGA of Indam (Italy)} \ and A.P. Rao
\vspace{6pt}\\
 {\small  Politecnico di Torino}\\
{\small\it  Corso Duca degli Abruzzi 24, 10129 Torino, Italy}\\
{\small\it e-mail: malaspina@calvino.polito.it}\\
\vspace{6pt}\\
{\small   University of Missouri-St. Louis}\\
 {\small\it  St Louis, MO, 63121, United States}\\
{\small\it e-mail: raoa@umsl.edu}} \maketitle \def\thefootnote{} 

\begin{abstract}We extend the  Horrocks  correspondence between vector bundles and cohomology modules on the projective plane
to the product of two projective lines. We introduce a set of invariants for a vector bundle on the product of two projective
lines, which includes the first cohomology module of the bundle, and prove that there is a one to one correspondence between
these sets of invariants and isomorphism classes of vector bundles without line bundle summands.
\end{abstract}

\begin{section}*{Introduction}
We denote by $\peen$ the $n$-dimensional projective space over a field $k$. For a vector bundle $E$ on $\peen$, its
intermediate cohomology modules are $H_*^i(E)=\bigoplus_{d\in \mathbb Z} H^i(E(d))$ ($0<i<n$), which are graded modules of
finite length over the polynomial ring in $n+1$ variables. Horrocks in \cite{Ho} established that a vector bundle on $\peen$
is determined up to isomorphism and up to a sum of line bundles ({\it i.e.} up to stable equivalence) by its collection of
intermediate cohomology modules and also a certain collection of extension classes involving these modules. In the simplest
case, on $\pee2$, this Horrocks correspondence states that stable equivalence classes of vector bundles are in one-to-one
correspondence with isomorphism classes of graded modules $M$ of finite length via the map $E \mapsto M = H^1_*(E)$. On
$\pee3$, each triple $(M_1, M_2, \rho)$, where $M_1, M_2$ are graded modules of finite length and $\rho$ is an element in
$\Ext^2(M_1^\vee, M_2^\vee)$,  corresponds to the stable equivalence class of a vector bundle $E$ with $M_1, M_2$ as its
intermediate cohomology modules.

This has led to attempts to study vector bundles on $\peen$ with restrictions placed on the intermediate cohomology modules.
Horrocks himself proved his splitting criterion ({\it loc. cit.}) which states that when all the intermediate cohomology
modules are zero, the bundle is a sum of line bundles. In general, for an $n$-dimensional projective variety $(X, \sO(1))$, a
vector bundle $E$ on $X$ with all intermediate cohomology modules $H^i_*(E), 0 < i < n$ equal to zero is called an ACM
(Arithmetically Cohen-Macaulay) bundle. Thus, on $\peen$, all ACM bundles are sums of line bundles of the form $\sO(a)$.
Refinements of this splitting criterion include the syzygy theorem for vector bundles on $\peen$ (\cite{Ein},\cite{E-G}) and
other splitting results for bundles with low ranks \cite{KPR}. Several papers study the restrictions on cohomology modules of
rank two bundles on $\pee3$ (\cite{Bur},\cite{Val},\cite{R}).

Horrocks' splitting criterion fails more generally. A subvariety  $(X, \sO(1))$ in projective space is called an
arithmetically Cohen-Macaulay variety if the embedding is projectively normal and $\sO$ itself is an ACM bundle. A smooth ACM
subvariety $X$ which is not a linear subspace will always have ACM vector bundles that are not equal to sums of line bundles
of the form $\sO(a)$ (see \cite{B-G-S} for smooth hypersurfaces in $\peen$ and \cite{KRR3} Proposition 5 for an argument that
works for any smooth ACM variety). On the $n$-dimensional quadric hypersurface $\Q_n$, the spinor bundles (see \cite{o1}) are
such ACM bundles and Kn\"{o}rrer's theorem \cite{Kn} states that all the ACM bundles on $\Q_n$ are direct sums of line
bundles and twists of spinor bundles. ACM bundles on other varieties have been studied by many authors, in particular on
hypersufaces (\cite{B-G-S},\cite{CH1},\cite{CH2},\cite{cf},\cite{CM1},\cite{CM2},\cite{CM3},\cite{fa2},\cite{KRR1},\cite{KRR2},\cite{Rav}),
on complete intersections \cite{BR},
on a few Fano threefolds (\cite{ac},\cite{af},\cite{bf1},\cite{fa1}) and on Grassmannians of lines (\cite{ag},\cite{cm00}).
\par
Ottaviani \cite{o2} has generalized the Horrocks splitting criterion to  quadrics and to Grassmannians. In the last few
years, extensions of the splitting criteria to other varieties have been made, as well as the cohomological characterization
of certain types of vector bundles on these varieties (\cite{am},\cite{bk},\cite{bm1},\cite{bm2},\cite{cm0}).
These results represent special cases
of what should be a Horrocks correspondence on other varieties. The aim of this paper is to obtain a Horrocks correspondence
on a product of two projective lines. In other words, we find cohomological invariants that determine the isomorphism class
of a vector bundle on the product of two projective lines.

Let $\Q$ be the smooth quadric in $\pee3$ defined by the equation $x_0x_3-x_1x_2$. Let $\{s,t\}, \{u,v\}$ be coordinates on
the factors of $\pee1 \times \pee1$. Then $\Q$ is identified with  $\pee1 \times \pee1$, with $x_0,x_1,x_2,x_3$ corresponding
to  $su,sv,tu,tv$. Let $\Sigma_i, i=1,2,$ be the pull back of $\OP 1(1)$ under the two projections to $\pee1$. These are the
spinor (line) bundles on $\Q$.

Let $E$ be a vector bundle on $\Q$ without any ACM summand, and let $M=M(E)$ denote $\bigoplus _{\nu\in \mathbb Z}H^1(\Q,
E(\nu))$, which is a graded module of finite length over the homogeneous coordinate ring $S(\Q)$. A free minimal presentation
of $M$ as an $S(\Q)$-module, when sheafified, yields an exact sequence of vector bundles on $\Q$, $0 \to F_M \to L_1\to
L_0\to 0$, where the $L_i$ are free and $H^1_*(F_M)=M$. Two auxiliary finite length $S(\Q)$-modules  $M_{\Sigma_i}=
H^1_*(F_M\otimes \Sigma_i)$ can hence be defined. We will show that $E$ determines a graded vector subspace $W$ inside the
$S(\Q)$-module $M_{\Sigma_1}$, where this subspace is annihilated by multiplication by $u,v$. Such a subspace $W$ will be
called a subspace of $\Sigma_2$-socle elements in $M_{\Sigma_1}$. Likewise, $E$ will determine a subspace $V$ in
$M_{\Sigma_2}$ consisting of $\Sigma_1$-socle elements.

Given the collection of such data $\{M, W\subset M_{\Sigma_1}, V\subset M_{\Sigma_2}\}$, where $M$ is a graded $S(\Q)$-module
of finite length, and $W, V$ are as described above, an isomorphism $M \to M'$ carries data $(M,W,V)$ to corresponding data
$(M', W', V')$. Hence we can talk about isomorphism classes of this data. Rather than consider stable equivalence classes of
vector bundles on $\Q$, we will consider isomorphism classes of vector bundles on $\sQ$ without ACM summands.

We prove (\ref{uniqueness}, \ref{existence})

\begin{theorem-nonumber} There is a bijection between isomorphism classes of vector
bundles without ACM summands on $\Q$ and isomorphism classes of triples $(M,W,V)$, where $M$ is a graded $S(\Q)$-module of
finite length, $W$ is a graded vector subspace of $\Sigma_2$-socle elements in $M_{\Sigma_1}$, $V$ is a graded vector
subspace of $\Sigma_1$-socle elements in $M_{\Sigma_2}$.
\end{theorem-nonumber}

We end the paper discussing examples that illustrate the invariants introduced. All extensions of Horrocks' ideas to other
varieties require an understanding of ACM bundles on the variety. In cases when there are only finitely many irreducible ACM
bundles (up to twist) on the variety, one can try to obtain a Horrocks correspondence as we have done here. In a paper in
preparation, we consider the case of higher dimensional quadric hypersurfaces.

\section {The $\gamma$-sequence of a bundle}
Let $\Q$ be the smooth quadric surface in $\pee3$ as described above. The spinor (line) bundles $\Sigma_1, \Sigma_2$ on $\Q$
satisfy the canonical sequences (pulled back from $\pee1$):
\begin{equation}
\begin{array}{ccccccc}
         0 \to \Sigma_1^{-1} \xrightarrow{[-t,s]^\vee} 2 \sO  \xrightarrow{[s,t]} \Sigma_1 \to 0, \\
        0 \to \Sigma_2^{-1} \xrightarrow{[-v,u]^\vee} 2 \sO  \xrightarrow{[u,v]} \Sigma_2 \to 0.
\end{array}
\end{equation} \label{spinor}

$\Sigma_1 \otimes \Sigma_2 = \sO(1)$ where $\sO(1)$ denotes the restriction of the hyperplane bundle of $\pee3$ to $\sQ$. We
will sometimes denote $\Sigma_1$ by $\sO(1,0)$, $\Sigma_2$ by $\sO(0,1)$ and $\sO(1)$ by $\sO(1,1)$, using the notation
$\sO(a,b)$ as well.

Let $S=k[x_0,x_1,x_2,x_3]$ be the polynomial ring and let $S(\Q)$ be the coordinate ring of $\Q$. We will always view $S(\Q)$
as a graded $k$-algebra with its grading inherited from $S$.
For any sheaf $F$ on $\Q$, we define $H^i_*(F)$ to be the graded $S(\Q)$-module $\bigoplus _{\nu\in \mathbb Z}H^i(\Q, F(\nu))$.

\begin{definition} Let $F$ be a sheaf on $\Q$.
\begin{enumerate}
    \item The graded $\Sigma_1$-socle sub-module of $H^i_*(F)$  is the kernel of
    the map $$H^i_*(F) \xrightarrow{[-t,s]^{\vee}}  H^i_*(2F\otimes \Sigma_1))$$.
    \item The graded $\Sigma_2$-socle sub-module of $H^i_*(F)$  is the kernel of
    the map $$H^i_*(F) \xrightarrow{[-v,u]^{\vee}}  H^i_*(2F\otimes \Sigma_2))$$.
\end{enumerate}
\end{definition}

If $F$ is a vector bundle on $\Q$, $H^1_*(F)$ and its $\Sigma_i$-socle sub-modules  are graded
finite length $S(\Q)$-modules. The $\Sigma_i$-socle sub-modules are annihilated by the forms $x_0,x_1,x_2,x_3$, hence
their $S(\Q)$-module structure descends to a graded $k$-vector space structure.

\begin{lemma} \label{socle} Let $F$ be a sheaf on $\Q$, $V$  a finite-dimensional graded subspace consisting of
$\Sigma_i$-socle elements in $H^1_*(F)$. Then there is a homomorphism $\alpha: V\otimes_k\Sigma_i^{-2} \to F$ such that
$H^1_*(\alpha)$ has image $V$. For any other such map $\alpha'$, there is an automorphism $\lambda$ of $V$ such that
$H^1_*(\alpha\circ (\lambda \otimes 1))= H^1_*(\alpha')$.
\end{lemma}
\begin {proof} Consider(for the case $\Sigma_i=\Sigma_1$) the canonical $\Sigma_1$ sequence (\ref{spinor}) tensored
by $F\otimes \Sigma_1$
$$0 \to  F \xrightarrow{[-t,s]^\vee} 2 F\otimes\Sigma_1  \xrightarrow{[s,t]} F\otimes\Sigma_1^2 \to 0.$$
There is a graded subspace $V'$ of $H^0_*(F\otimes\Sigma_1^2 )$ which is mapped isomorphically to $V \subset H^1_*(F)$. This
induces a map $\alpha: V'\otimes_k\sO \to F\otimes\Sigma_1^2 $ and thus the commuting diagram
$$
\begin{array}{cccccccc}

0 \to & F &\xrightarrow{[-t,s]^\vee} &2 F\otimes\Sigma_1  &\xrightarrow{[s,t]} &F\otimes\Sigma_1^2 \to 0 \\
    &\uparrow\alpha   &            &  \uparrow{\alpha \oplus \alpha} &         & \uparrow\alpha &\\
0 \to &V'\otimes_k\Sigma_1^{-2} &\xrightarrow{[-t,s]^\vee} &2 V'\otimes_k\Sigma_1^{-1}  &\xrightarrow{[s,t]} &V'\otimes_k\sO
\to 0
\end{array}
$$
Then $H^1_*(\alpha): H^1_*(V'\otimes_k\Sigma_1^{-2}) \to H^1_*(F)$ gives $V' \cong V$.

Any other map $\alpha'$ inducing a similar isomorphism must fit into a similar commutative diagram, with some $V''\subset
H^0_*(F\otimes\Sigma_1^2 )$. The map $H^1_*(\alpha)^{-1}\circ H^1_*(\alpha'): V'' \cong V \cong V'=H^1_*(V'\otimes_k
\Sigma_1^{-2})$ is likewise induced by some $\lambda: V''\otimes_k\sO \to V'\otimes_k\sO$ and it easy to verify that
$\lambda$ is
 an isomorphism.
 \end{proof}

For any vector bundle $E$ on $\Q$, we define its first cohomology module $M(E)= H^1_*(E)$. For $\sO$, $\Sigma_i$, $\Sigma_i
^{-1}$, this module is zero, hence these three are ACM line bundles on $\Q$. Using the K\"unneth formula on $\pee1\times
\pee1$, we can see that, as graded vector spaces, $M(\Sigma_i^{-2}) =k$ , $M(\Sigma_i^{-3})= k^{\oplus 2} \oplus k^{\oplus
2}$ (two components in degrees 0 and 1) .

\begin{definition} A vector bundle $L$ on $\Q$ will be called free if $L\cong \bigoplus_{i=1}^n\sO(a_i), a_i \in \mathbb Z$.
A vector bundle $E$ on $\Q$ has no ACM summands if it cannot be expressed as $E'\oplus L$ where $L$ is some line bundle of
the form $\sO(a),\Sigma_1(b),\Sigma_2(c)$.
\end{definition}

Consider a finitely generated graded $S(\Q)$-module $M$. It has a minimal free
presentation as an $S(\Q)$-module, given by a homogeneous map between free modules, where minimality means that the matrix of
the map has no entry equal to a unit and no column equal to zero. This minimal presentation is unique up to isomorphism and
any other free presentation of $M$ will split as the direct sum of the
minimal presentation and a complex of free modules with the zero map.

In the special case when $M$ is a graded $S(\Q)$-module of finite length, we may sheafify the minimal presentation to  get a
sequence of vector bundles on $\Q$ (since $M$ gives the zero sheaf on $\Q$):
$$ \Psi: 0 \to F_M \to L_1 \to L_0 \to 0,$$
where $L_1, L_0$ are free bundles, and $F_M$, unique up to isomorphism, is locally free on $\Q$, with $H^1_*(F_M) = M$. We will
call $F_M$ the vector bundle associated to $M$.

\begin{definition} Let $M$ be a graded $S(\Q)$-module of finite length. Define $M_{\Sigma_i}$ to be the
(finite length) graded $S(Q)$-module $H^1_*(F_M \otimes \Sigma_i)$, where $F_M$ is the vector bundle associated to $M$.
\end{definition}

 Let $E$ be a vector bundle on $\Q$. Since the invariant $M= H^1_*(E)$ is a graded $S(\Q)$-module  of finite length, we can
 attach to  $E$ the invariants $M, F_M,  M_{\Sigma_1}, M_{\Sigma_2}$.  The short exact sequence obtained from the minimal
 presentation of  $M$,
 $$ \Psi: 0 \to F_M \to L_1 \xrightarrow{\psi} L_0 \to 0,$$
 gives an  extension class $\Psi\in \Ext^1(L_0,F_M)\cong H^1(L_0^\vee \otimes F_M)$.

 The construction of $\Psi$ started with a minimal presentation of $M$, hence a surjective homomorphism of $S(\Q)$-modules
 $H^0_*(L_0)\to M$. Following Horrocks \cite{Ho} (see also \cite{BH}), we have a second construction.
 From the homomorphism $H^0_*(L_0)\to M$, we get an element (of degree 0) in $H^0_*(L_0)^{\vee}
 \otimes M \cong H^1_*(L_0^{\vee} \otimes E)\cong \Ext^1(L_0,E)$.
 Hence we get an exact sequence
 $$ \gamma: 0 \to E \xrightarrow {f} A_E \xrightarrow {g} L_0 \to 0,$$
 whose connecting homomorphism $H^0_*(L_0)\to H^1_*(E)$ returns the original map $H^0_*(L_0)\to M$. Therefore the bundle
 $A_E$ in the middle of $\gamma$ has $H^1_*(A_E)=0$. Thus $A_E$ is an ACM bundle $A_E$.

 We call this sequence a $\gamma$-sequence for $E$, where $A_E$ is ACM and $L_0$ is free, with $H^0_*(L_0)$ minimally
 surjecting onto  $H^1_*(E)$. There  are natural actions of $\Aut(L_0)$ and $\Aut(E)$ on the
 collection of such extensions $\gamma$ for a fixed $E$.

\begin{lemma}\label{liftings} Let $E,E'$ be bundles  on $\Q$, with $\gamma$-sequences $\gamma: 0 \to E \xrightarrow {f} A_E
\xrightarrow {g}L_0 \to 0$, and $ \gamma': 0 \to E' \xrightarrow {f'} A_{E'} \xrightarrow {g'}L_0' \to 0 $ .
\begin{enumerate}
\item Any homomorphism $\sigma: E \to E'$ gives an induced homomorphism $M(E)\to M(E')$ and an induced a map of exact
sequences $\gamma \to \gamma'$ \ along with well-defined commutative diagrams
 \[
\begin{CD} H^1(E\otimes \Sigma_i)& \xrightarrow {f} & H^1_*(A_E\otimes \Sigma_i)&\to 0\\
            \downarrow &               & \downarrow &\\
    H^1(E'\otimes \Sigma_i)& \xrightarrow {f'} &H^1_*(A_{E'}\otimes \Sigma_i)&\to 0
\end{CD} \] for $i=1,2$.
\item If $\sigma$ induces an isomorphism $M(E)\to M(E')$, then $L_0\cong L_0'$ and modulo the action of $\Aut(L_0)$, we may
take the map $\gamma \to \gamma'$ of exact sequences to be the push-out of $\gamma$ by $\sigma$.
\end{enumerate}
\end{lemma}
\begin{proof} The first follows from the exact sequence $$0 \to \Hom(L_0,A_{E'})\xrightarrow{g} \Hom(A_E,A_{E'})\xrightarrow{f}
\Hom(E,A_{E'})\to \Ext^1(L_0,A_{E'}),$$ where the last group is zero. In the resulting diagram
\[
\begin{CD} 0 \to & E &\xrightarrow {f}   &A_E &    \xrightarrow {g} &L_0& \to 0\\
                  &  \downarrow{ \sigma}&  &\downarrow{ \sigma_1}& & \downarrow{ \sigma_2}&\\
           0 \to &E'& \xrightarrow {f'}  &A_{E'}&\xrightarrow {g'} &L_0'& \to 0
           \end{CD}\]
the induced map $\sigma_1:A_E \to A_{E'}$ is well-defined up to a map factoring through $L_0$, which would not change the map
$\sigma_1:H^1_*(A_E\otimes \Sigma_i) \to H^1_*(A_{E'}\otimes \Sigma_i)$.
\par
For the second part, since both $L_0, L_0'$ arise from minimal resolutions of the same module $M(E)\cong M(E')$, they are
isomorphic, and furthermore $\sigma_2: H^0_*(L_0)\to H^0_*(L_0')$ is  an isomorphism since it induces the isomorphism
$\sigma:M(E)\to M(E')$.
\end{proof}

\begin {corollary}\label{c1} Let $E$ be a bundle on $\Q$ and let $M=M(E)$ and $F$ be the vector bundle associated to the
module $M$. Then there is a homomorphism $\beta: F \to E$ such that  $\gamma$ is the push-out by $\beta$ of  $\Psi$ (modulo $\Aut(L_0)$).
\end{corollary}
\begin{proof} Since both
$H^0_*(\psi)$ and $H^0_*(g)$ present $M$ and since $H^0_*(L_1)$ is free, we can find a lifting $L_1 \to A_E$ making the
diagram commute. This gives a map $\beta:F \to E$ inducing an isomorphism on $H^1_*$.
 \end{proof}

The following proposition obtains conditions for comparing extensions $\gamma$ and $\gamma'$ for two bundles $E$ and $E'$.

 \begin{proposition}\label{gamma-sequences} Let $E,E'$ be two vector bundles on $\Q$ with $M(E)\cong M(E')\cong M$.
 Let $F$ be the vector bundle
 associated to the module $M$, and let $\beta: F \to E, \beta': F\to E'$ be morphisms for which $\gamma, \gamma'$ are push-outs
 of $\Psi$. If $\ker [1\otimes \beta: H^1_*(\Sigma_i \otimes F) \to H^1_*(\Sigma_i \otimes E)]$ equals
 $\ker [1\otimes \beta': H^1_*(\Sigma_i\otimes F) \to H^1_*(\Sigma_i \otimes E')]$ for $i=1,2$, then there
 exists $\sigma: E \to E'$ such that $\gamma'$ is
 the push-out of $\gamma$ by $\sigma$, and $\sigma$ induces $M(E)\cong M(E')$.
 \end{proposition}
 \begin{proof}

The extension class $\gamma$ can be viewed (with abuse of notation) as the element in $H^1(L_0^\vee \otimes E)$ given by
$\delta(\gamma)(Id_{L_0})$ where $\delta(\gamma): H^0(L_0^\vee \otimes L_0) \to H^1(L_0^\vee \otimes E)$ is the connecting
homomorphism, which is also (up to sign) the image of $Id_E$ in the connecting homomorphism $\delta(\gamma^\vee): H^0(E^\vee
\otimes E) \to H^1(L_0^\vee
\otimes E)$ for the dual short exact sequence $\gamma^\vee$.\\
Consider the commuting diagram derived from $\gamma^\vee$:
\[ \begin{CD}
H^0(E^\vee \otimes E) \xrightarrow{\delta(\gamma^\vee)} & H^1(L_0^\vee \otimes E) &\xrightarrow{g^\vee\otimes 1}
&H^1(A_E \otimes E)\\
& \uparrow 1 \otimes \beta &  & \uparrow  1 \otimes \beta  \\
& H^1(L_0^\vee\otimes F)  &\xrightarrow{ g^\vee\otimes 1} &H^1(A_E\otimes F). \\
\end{CD}
\]

The element $\Psi \in H^1(L_0^\vee\otimes F)$ is pushed out by $\beta$ to the extension class $\gamma$ which is
$\delta(\gamma^\vee)(I_E)$. Hence $(1 \otimes \beta)\circ (g^\vee\otimes 1)(\Psi)= (g^\vee\otimes 1)\circ (1 \otimes
\beta)(\Psi)=0$. Since $A_E$ is ACM, $1\otimes \beta:H^1(A_E\otimes F) \to H^1(A_E \otimes E)$ is a direct sum of maps
$H^1(\Sigma_i(t)\otimes F) \to H^1(\Sigma_i(t)\otimes E)$ and $H^1(\sO(t)\otimes F) \to H^1(\sO(t)\otimes E)$, $i=1,2, t\in
\mathbb Z$. Since the last maps are isomorphisms, any non-zero components of $(g^\vee\otimes 1)(\Psi)$ in $H^1(A_E\otimes F)$
lie only in the summands of type $H^1(\Sigma_i(t)\otimes F)$  and they will be in the kernel of
$ 1\otimes \beta: H^1(\Sigma_i(t) \otimes F) \to H^1(\Sigma_i(t) \otimes E)$.

Consider the commuting diagram:

\[ \begin{CD}
H^0(E^\vee \otimes E\sp{\prime}) \xrightarrow{\delta(\gamma^\vee)} & H^1(L_0^\vee \otimes E\sp{\prime})
&\xrightarrow{g^\vee\otimes 1} &H^1(A_E \otimes E\sp{\prime})\\
& \uparrow 1 \otimes \beta\sp{\prime} &  & \uparrow  1 \otimes \beta\sp{\prime}  \\
& H^1(L_0^\vee\otimes F)  &\xrightarrow{ g^\vee\otimes 1} &H^1(A_E \otimes F).\\
\end{CD}
\]

By corollary \ref{c1}, $(1 \otimes \beta\sp{\prime})(\Psi)$  represents $\gamma'$, hence is non-zero.  By our assumptions
about the kernels of $1\otimes \beta$ and $1\otimes \beta'$ on $H^1_*(\Sigma_i\otimes F)$, $(1 \otimes
\beta\sp{\prime})\circ(g^\vee\otimes 1)(\Psi)=0$.  Thus $(1 \otimes \beta\sp{\prime})(\Psi)$ is the image of some element
$\sigma \in H^0(E^\vee \otimes E\sp{\prime})$, necessarily non-zero.

Consider the map $\sigma: E \to E'$. Since $\sigma$ maps to $\gamma'$ under $\delta(\gamma^\vee)$, it follows that the
pushout of $\gamma$ under $\sigma$ is $\gamma'$,
\[ \begin{CD}
0 \to & E \to &A_E \to &L_0 &\to 0\\
      & \downarrow \sigma & \downarrow \sigma_1 & ||&\\
0 \to & E' \to &A_{E'} \to &L_0 &\to 0.
\end{CD}
\]

The induced map $H^1_*(E) \xrightarrow{\sigma} H^1_*(E')$ fits into the commuting diagram
\[ \begin{CD}
H^0_*(L_0) \to & H^1_*(E) \to 0 \\
||              & \downarrow \sigma \\
H^0_*(L_0) \to & H^1_*(E') \to 0.
\end{CD}
\]
Hence $H^1_*(E) \xrightarrow{\sigma} H^1_*(E')$ must be surjective, and since they are both isomorphic to the module $M$
which is a finite dimensional vector space over the base field, $H^1_*(E) \xrightarrow{\sigma} H^1_*(E')$ is an isomorphism.
\end{proof}

\end{section}

\begin{section}{The diagram of a bundle}
Let $E$ be a vector bundle on $\sQ$. From the previous section, we get a graded $S(Q)$-module $M(E)$, a bundle $F$ associated
to $M(E)$ and an ACM bundle $A_E$ with a push-out diagram

\[
\begin{CD} \Psi:\qquad &0 \to & F &\xrightarrow {}   &L_1&    \xrightarrow {} &L_0& \to 0\\
                 & &  \downarrow{ \beta}&  &\downarrow{ \beta_1}& & \downarrow{ \beta_2}&\\
           \gamma:\qquad &0 \to &E& \xrightarrow {f}  &A_{E}&\xrightarrow {g} &L_0& \to 0.
           \end{CD}\]

Suppose $E = E'\oplus L$ where $L$ is some line bundle of the form $\sO(a),\Sigma_1(b),\Sigma_2(c)$. Then applying Lemma
\ref{liftings} to the projection $\sigma: E\to E'$ we see that $$\gamma = \gamma' \oplus \{0 \to L \xrightarrow {Id_L} L \to
0 \to 0\}.$$ Hence we can ``minimize $\gamma$'' by choosing $E$ to have no ACM summands.

 By Kn\"{o}rrer's theorem \cite{Kn}, the ACM bundle $A_E$ has the form
  $$A_E\cong [\bigoplus_{i\in \mathbb Z} \mu_i \Sigma_1(i)]\oplus\ [\bigoplus_{j\in \mathbb Z} \nu_j
 \Sigma_2(j)]\oplus P_E = S_E \oplus P_E$$ where $P_E$ is free.\\
We  define the following numerical invariants associated to $E$ which determine $S_E$:
$$\text{$\mu=\{\mu_i\}_{i\in \mathbb Z}$ and   $\nu = \{\nu_j\}_{j\in \mathbb Z}$.}$$
Furthermore, let $V$ be a graded vector space with dimensions $\mu$, $W$ with dimensions $\nu$. Then $$A_E\cong S_E \oplus
P_E \cong (V\otimes_k \Sigma_1) \oplus (W\otimes_k \Sigma_2) \oplus P_E.$$

Using the canonical spinor sequences (\ref{spinor}), we get a canonical exact sequence $\Delta$:
$$ \Delta: \qquad 0 \to K_{E} \to L_E \oplus P_E \xrightarrow{} A_E \to 0, $$
 where $K_{E} = (V \otimes \Sigma_1^{-1})\oplus \ (W \otimes \Sigma_2^{-1})$ and
 $L_E = (V \otimes  2\sO)\oplus (W \otimes  2\sO).$

\begin{proposition}\label{diagram} There is a commuting diagram
 $$\begin{array}{ccccccccc}
 & & & \eta & &\Delta&  & & \\
 \\
 & & &0 & &0& & & \\
 & &  &\downarrow & &\downarrow & & &\\
 & & & K_E                &=&K_E   &    &  & \\
 & & &\downarrow\alpha & &\downarrow & & & \\
 & & 0 \to &F\oplus Q_E &\rightarrow & L_E \oplus P_E & \xrightarrow {}& L_0& \to 0\\
 & & &\downarrow\beta & &\downarrow & &\| &\\
 & \gamma:\qquad & 0 \to &E& \xrightarrow{f} & A_E&\xrightarrow{g} & L_0&  \to 0\\
 & & &\downarrow & &\downarrow & & & \\

 & & & 0&          &0& & & \\
 \end{array}$$
where $Q_E$ is free and the induced map $H^1_*(F)\to H^1_*(E)$ is an isomorphism. Furthermore, if $E$ has no ACM summands, then
\begin{enumerate}
\item $H^1(\alpha \otimes
Id_{\Sigma_i}): H_*^1(K_E\otimes \Sigma_i) \to H^1_*(F\otimes \Sigma_i)$ is injective for $i=1,2$. \item the image of
$H^1(\alpha \otimes Id_{\Sigma_1})$ is a graded vector subspace of $H^1_*(F\otimes \Sigma_1) = M_{\Sigma_1}$ isomorphic to
$W(1)$ and consists of $\Sigma_2$-socle elements. \item the image of $H^1(\alpha \otimes Id_{\Sigma_2})$ is a graded vector
subspace of $H^1_*(F\otimes \Sigma_2) = M_{\Sigma_2}$ isomorphic to $V(1)$ and consists of $\Sigma_1$-socle elements.
\end{enumerate}
\end{proposition}

\begin{proof} Since $H^0_*(L_E \oplus P_E)$ is a free $S(\Q)$-module and surjects onto $H^0_*(A_E)$, we obtain a composite map
$H^0_*(L_E \oplus P_E) \to H^0_*(L_0)$ that gives a free presentation for $M(E)$ and hence can be compared with the
minimal presentation $\Psi$ for $M(E)$.
This gives the second row of the diagram where $Q_E$ is some free bundle. Since the bottom row is a push-out of the second row,
$H^1_*(F)\to H^1_*(E)$ is an isomorphism. The sequence labelled $\eta$ is the pull-back by $f$ of the sequence $\Delta$.

Now suppose $E$ has no ACM summands. We wish to show that $H^0_*(E\otimes \Sigma_i)
\to H_*^1(K_E\otimes \Sigma_i)$ is the zero map.
\par
Indeed, suppose that when $i=1$, a section $s$ of $E\otimes \Sigma_1(p)$ maps to a nonzero element of $H^1(K_E\otimes
\Sigma_1(p))$. Then the image $s'$ of $s$ in $H^0(A_E\otimes \Sigma_1(p))$  maps to a nonzero element of $H^1(K_E\otimes
\Sigma_1(p))$. But the second column is built out of a direct sum of the canonical sequences (\ref{spinor}) (twisted by
integers $q$).  Tensoring such a sequence with $\Sigma_1(p)$ gives $$0 \to \Sigma_j^{-1}(q)\otimes \Sigma_1(p) \to 2 \sO (q)
\otimes \Sigma_1(p) \to \Sigma_j (q) \otimes \Sigma_1(p) \to 0.$$ The induced map from $H^0$ to $H^1$ is nonzero only if
$j=2, p+q=-1$, and in this case, the element $1$ in $H^0( \Sigma_2 \otimes \Sigma_1(-1))=H^0(\sO)$ maps to a nonzero element.

Hence there is a summand of $A_E\otimes \Sigma_1(p)$ of the form $\sO$, such that $s'$ projects to the global section $1$ of
this summand. Thus  $E\otimes \Sigma_1(p)$ has  $\sO$ as a summand, giving $\Sigma_2(p-1)$ as an ACM summand of $E$.

\par
The last statements follow  from the the fact that (for $\Sigma_1$) $$H^1_*(K_E\otimes \Sigma_1) = H^1_*(W\otimes_k
\Sigma_2^{-1}\otimes \Sigma_1) =H^1_*(W(1)\otimes_k \sO(0,-2))=W(1)$$ and from the commutative diagram
\[
\begin{CD}
    K_E\otimes \Sigma_1 &\longrightarrow &F\otimes \Sigma_1 \\
         \downarrow \bmatrix -v\\u \endbmatrix & &\downarrow \bmatrix -v\\u \endbmatrix\\
      2K_E(1)           &\longrightarrow &2F(1)  \\
\end{CD}
\]
where $H^1_*(K_E)=0$.

\end{proof}

\begin{definition}The diagram in Proposition \ref{diagram} will be called a diagram of $E$.
\end{definition}

\begin{proposition}\label{4-term-sequence} For $i=1,2$, there is an exact sequence
$$ 0 \to H^1_*(K_E\otimes \Sigma_i) \xrightarrow{\alpha} H^1_*(F\otimes \Sigma_i)\xrightarrow{\beta} H^1_*(E\otimes \Sigma_i)
\xrightarrow{f} H^1_*(A_E\otimes \Sigma_i) \to 0$$ obtained from the diagram of $E$.
\end{proposition}
\begin{proof} It suffices to verify exactness at $H^1_*(E\otimes \Sigma_1)$. The diagram
\[ \begin{CD}
&2H^0_*(L_0) &\rightarrow  &2H^1_*(E) \to 0 \\
 &\downarrow {s,t}    &          &\downarrow {s,t} \\
 &H^0_*(L_0\otimes\Sigma_1)  &\to       &H^1_*(E\otimes\Sigma_1) &\xrightarrow{f} H^1_*(A_E\otimes \Sigma_1)\to 0 \\
 & \downarrow & & \\
 & 0 & &\\
 \end{CD}
 \]
shows that $\ker f$ equals the image of $2H^1_*(E) \xrightarrow{s,t} H^1_*(E\otimes\Sigma_1)$. As a special case, when $E$ is
$F$, it follows that $2H^1_*(F) \xrightarrow{s,t} H^1_*(F\otimes\Sigma_1)$ is onto. Now the diagram
\[ \begin{CD}
& 2H^1_*(F) & \xrightarrow{\cong}& 2H^1_*(E) \\
&\downarrow {s,t}    &          &\downarrow {s,t} \\
& H^1(F\otimes \Sigma_1) & \xrightarrow{\beta}     & H^1_*(E\otimes\Sigma_1)\\
& \downarrow & & \\
 & 0 & &\\
 \end{CD}
 \]
 shows that $\ker f$ is the image of $\beta$.

\end{proof}

The following theorem compares extensions $\eta: 0 \to K_E \xrightarrow{\alpha}F\oplus Q_{E} \xrightarrow{\beta} E \to 0$ and
$\eta': 0 \to K_{E'} \xrightarrow{\alpha'}F'\oplus Q_{E'} \xrightarrow{\beta'} E' \to 0$ for two bundles $E$ and $E'$.

\begin{theorem}\label{eta-sequences} Let $E,E'$ be two  vector bundles on $\Q$ with no ACM summands.
Suppose $M(E)= M(E')$, $K_E = K_{E'}$. In the diagram of proposition \ref{diagram}, suppose for $i=1,2$, $H^1(\alpha \otimes
Id_{\Sigma_i})$ and $H^1(\alpha' \otimes Id_{\Sigma_i})$ have the same image in $H^1_*(F\otimes \Sigma_i)$. Then there is a
homomorphism $\sigma: E \to E'$ such that the extension class $\eta$ is the pull-back by $\sigma$ of the extension $\eta'$
(up to an automorphism of $K_E = K_{E'}$).
\end{theorem}

\begin{proof}

 With $K_E = K_{E'} =K$, we have the two commuting diagrams
\[ \begin{CD}
H^0(E^\vee \otimes E) \xrightarrow{\delta(\eta)} & H^1(E^\vee\otimes K) &\xrightarrow{1\otimes \alpha} &H^1(E^\vee\otimes
(F\oplus Q_E))\\
& \uparrow (f)^\vee \otimes 1 &  & \uparrow  (f)^\vee \otimes 1 \\
& H^1(A_E^\vee\otimes K)  &\xrightarrow{1\otimes \alpha} &H^1(A_E^\vee\otimes (F\oplus Q_E))
\end{CD}
\]

and
\[
\begin{CD}
H^0(E^\vee \otimes E') \xrightarrow{\delta(\eta')} & H^1(E^\vee\otimes K) &\xrightarrow{1\otimes \alpha'}
&H^1(E^\vee\otimes (F\oplus Q_{E'}))\\
& \uparrow (f)^\vee \otimes 1 &  & \uparrow  (f)^\vee \otimes 1 \\
& H^1(A_E^\vee\otimes K)  &\xrightarrow{1\otimes \alpha'} &H^1(A_E^\vee\otimes (F\oplus Q_{E'})).
\end{CD}
\]

 Let $K = V\otimes \Sigma_1^{-1} \oplus W\otimes \Sigma_2^{-1}$ for graded vector spaces $V,W$.  By the assumption on
  the images of $H^1_*(K\otimes \Sigma_i)$ in $H^1_*(F\otimes \Sigma_i)$ and using lemma \ref{socle}, we can
 assume there are automorphisms  $\lambda_1,\lambda_2$ of $V,W$ such that if $\lambda = \bmatrix \lambda_1 & 0 \\ 0 &
 \lambda_2 \endbmatrix: K \to K$, then the maps
 $H^1(1\otimes \alpha')$ and $H^1(1 \otimes (\alpha \circ \lambda))$ from $H^1_*(\Sigma_i^{-1}\otimes K)$ to
 $H^1_*(\Sigma_i^{-1}\otimes (F\oplus Q_{E'}))$ are equal. (note that $H^1_*(\Sigma_i^{-1}\otimes
 Q_{E'})=0$.)
 Since $A_E^\vee$ is a direct sum of
 terms of the form $\sO(t), \Sigma_i^{-1}(t)$, it follows that the maps $H^1(1\otimes \alpha')$ and $H^1(1 \otimes
 (\alpha \circ \lambda))$ from $H^1_*(A_E^\vee\otimes K)$ to $H^1_*(A_E^\vee\otimes (F\oplus Q_{E'}))$ are equal.

 In the first diagram, the element $Id_E \in H^0(E^\vee \otimes E)$ maps to $\eta \in H^1(E^\vee\otimes K)$. Hence
 $(1\otimes \alpha)\circ ((f)^\vee \otimes  1)(\Delta)=0$. Thus also $((f)^\vee \otimes 1)\circ (1\otimes \alpha)(\Delta)=0$.

  Hence $((f)^\vee \otimes 1)\circ (1\otimes \alpha')
 (1\otimes \lambda)^{-1}(\Delta)=0$
 and thus  $(1\otimes \alpha')\circ (f)^\vee \otimes 1)((1\otimes\lambda)^{-1}(\Delta)=0$. $(1\otimes \lambda)^{-1}(\Delta)$
 is non-zero, and
 $(f)^\vee \otimes 1$ is injective on
 $H^1(A_E^\vee\otimes K)$, hence there is a non-zero $\sigma \in H^0(E^\vee \otimes
 E')$ with the property that the pull-back by $\sigma$ of the sequence $\eta'$ is $(1\otimes\lambda)^{-1}(\eta)$. Thus
\[
\begin{CD}
0 \to &K \xrightarrow{\alpha'}&F\oplus Q_{E'} \xrightarrow{\beta'} &E' \to 0 \\
    &\uparrow \lambda                         & \uparrow\tilde\sigma                  & \uparrow\sigma \\
0 \to &K \xrightarrow{\alpha}&F\oplus Q_{E} \xrightarrow{\beta} &E \to 0.
\end{CD}
\]

\end{proof}

\end{section}

\begin{section} {Invariants of a bundle}

We have seen that a vector bundle $E$ on $\Q$ with no ACM summands determines the following invariants.
\begin{itemize}
    \item A graded $S(\Q)$-module $M=H^1_*(E)$ of finite length, hence an associated bundle $F$ and associated
    modules $M_{\Sigma_1}$, $M_{\Sigma_2}$.
    \item Finitely supported sequences $\mu = \{\mu_i\}_{i\in \mathbb Z}$, $\nu=\{\nu_j\}_{j\in \mathbb Z}$ of
    positive integers (determined by $A_E$). Hence graded vector spaces $V,W$ with graded dimensions determined
    by $\mu,\nu$.
    \item In the module $M_{\Sigma_1}$, a graded subspace of $\Sigma_2$ -socle elements
    isomorphic to $W(1)$.
    \item In the module $M_{\Sigma_2}$, a graded subspace of $\Sigma_1$ -socle elements
    isomorphic to $V(1)$.
\end{itemize}

Any bundle in the stable equivalence class of $E$ determines the same data as $E$. Replacing $E$ by $E(e)$ changes this data
by a corresponding shift in grading.

\begin{lemma} Given an isomorphism $\phi: M \to M'$, of graded $S(\Q)$-modules of finite length. There are well-defined induced
isomorphisms $M_{\Sigma_i}\cong M_{\Sigma_i}'$.
\end{lemma}
\begin{proof} Indeed, given the map $\phi: M \to M'$, there is an induced map from the sequence $\Psi: 0 \to F \to L_1 \to L_0
\to 0$ to the corresponding sequence $\Psi': 0 \to F' \to L_1' \to L_0' \to 0$. The map $F\to F'$ is an isomorphism, unique
up to a term that factors through $L_1$. Hence the induced map $H^1_*(F\otimes \Sigma_i) \to H^1_*(F'\otimes \Sigma_i)$ is
well-determined by $\phi$.
\end{proof}

\begin{theorem}\label{uniqueness}(\textbf{Uniqueness})\
Given $E,E'$ two bundles on $\Q$ without ACM summands, with invariants $M, W(1)\subset M_{\Sigma_1}, V(1)\subset
M_{\Sigma_2}$ and $M', W'(1)\subset M_{\Sigma_1}', V'(1)\subset M_{\Sigma_2}'$.
 Suppose $\exists \phi: M \cong M'$, such that the induced isomorphisms $M_{\Sigma_i}\cong
M_{\Sigma_i}'$ carry $V(1)$ to $V'(1)$ and $W(1)$ to $W'(1)$. Then $E$ and $E'$ are isomorphic.
\end{theorem}
\begin{proof}
We may assume that $M=M'$ and that there is an automorphism $\phi$ of $F$ which carries $V(1)$ to $V'(1)$ and $W(1)$ to
$W'(1)$. Since the graded vector spaces $V,W$ are isomorphic to $V',W'$, we may assume that $K_E=K_{E'}=K$. In the diagram of
proposition \ref{diagram} for $E'$, we may replace $\alpha'$ by $\phi^{-1}\circ \alpha'$ \textit{etc.} and assume that
$\alpha$ and $\alpha'$ give the same images in $H^1_*(F\otimes \Sigma_i)$. By applying Theorem \ref{eta-sequences}, there
exists $\sigma: E \to E'$ and $\lambda: K \cong K$ such that there is a commutative diagram
\[
\begin{CD}
0 \to &K \xrightarrow{\alpha'}&F\oplus Q_{E'} \xrightarrow{\beta'} &E' \to 0 \\
    &\uparrow \lambda                         & \uparrow\tilde\sigma                  & \uparrow\sigma \\
0 \to &K \xrightarrow{\alpha}&F\oplus Q_{E} \xrightarrow{\beta} &E \to 0.
\end{CD}
\]
This gives a diagram of cohomology
\[\begin{CD}
0 \to &H^1_*(K\otimes \Sigma_i) \xrightarrow{\alpha'}&H^1_*(F\otimes \Sigma_i) \xrightarrow{\beta'}
&H^1_*(E'\otimes \Sigma_i) \to &H^2_*(K\otimes \Sigma_i)  \\
    &\uparrow\lambda                         & \uparrow\tilde\sigma                  & \uparrow\sigma    &\uparrow\lambda\\
0 \to &H^1_*(K\otimes \Sigma_i) \xrightarrow{\alpha}&H^1_*(F\otimes \Sigma_i) \xrightarrow{\beta}
&H^1_*(E\otimes \Sigma_i) \to &H^2_*(K\otimes \Sigma_i)  \\
\end{CD}
\]
where since $\lambda: H^2_*(K\otimes \Sigma_i) \to H^2_*(K\otimes \Sigma_i)$ is an isomorphism, we get $\coker \beta
\hookrightarrow \coker \beta'$.

By Lemma \ref{liftings}, $\sigma$ also induces
\[
\begin{CD} \gamma':\qquad  &0 \to & E' &\xrightarrow {f'}   &A_{E'} &    \xrightarrow {g'} &L_0& \to 0\\
                 & &  \uparrow{ \sigma}&  &\uparrow{ \sigma_1}& & \uparrow{ \sigma_2}&\\
           \gamma':\qquad &0 \to &E& \xrightarrow {f}  &A_{E}&\xrightarrow {g} &L_0& \to 0
           \end{CD}\]

By Lemma \ref{4-term-sequence}, we get
\[\begin{CD}
0 \to &H^1_*(K\otimes \Sigma_i) \xrightarrow{\alpha'}&H^1_*(F\otimes \Sigma_i) \xrightarrow{\beta'}
&H^1_*(E'\otimes \Sigma_i) \to &H^1_*(A_{E'}\otimes \Sigma_i) \to 0 \\
    &\uparrow\lambda                         & \uparrow\tilde\sigma                  & \uparrow\sigma    &\uparrow\sigma_1\\
0 \to &H^1_*(K\otimes \Sigma_i) \xrightarrow{\alpha}&H^1_*(F\otimes \Sigma_i) \xrightarrow{\beta}
&H^1_*(E\otimes \Sigma_i) \to &H^1_*(A_E\otimes \Sigma_i) \to 0 \\
\end{CD}
\]

Hence, $\coker \beta$ and $\coker \beta'$ can be identified with $H^1_*(A_E\otimes \Sigma_i)=H^1_*(A_{E'}\otimes \Sigma_i)$.
Therefore,
$\sigma_1:H^1_*(A_E\otimes \Sigma_i) \to H^1_*(A_{E'}\otimes \Sigma_i)$ is an inclusion of vector spaces of the same rank,
hence is an isomorphism.

Thus there exists $\sigma: E \to E'$ such that the induced map $\sigma_1:H^1_*(A_E\otimes \Sigma_i) \to H^1_*(A_{E'}\otimes
\Sigma_i)$ described in Lemma \ref{liftings} is an isomorphism for each $i$.  Again, by Theorem \ref{gamma-sequences}, there
exists $\sigma: E \to E'$ such that $\sigma$ induces $M(E)\cong M(E')$ (and $\gamma'$ is
 the push-out of $\gamma$ by $\sigma$). By the Zariski openness of these conditions, the general $\sigma: E \to E'$ has the
 property that in the induced diagram
 \[
\begin{CD} \gamma':\qquad  &0 \to & E' &\xrightarrow {f'}   &A_{E'} &    \xrightarrow {g'} &L_0& \to 0\\
                 & &  \uparrow{ \sigma}&  &\uparrow{ \sigma_1}& & \uparrow{ \sigma_2}&\\
           \gamma':\qquad &0 \to &E& \xrightarrow {f}  &A_{E}&\xrightarrow {g} &L_0& \to 0,
           \end{CD}\]
$\sigma_2$ is an isomorphism, and $\sigma_1:H^1_*(A_E\otimes \Sigma_i) \cong H^1_*(A_{E'}\otimes \Sigma_i)$ for $i=1,2$.\\
The map $\sigma_1 : A_E= S \oplus P_E \to A_{E'}= S\oplus P_{E'}$ can be written as $\bmatrix \sigma_{11} & \sigma_{12}\\
\sigma_{21} & \sigma_{22}\endbmatrix$. Then $\sigma_{11}$ has the property that $H^1(\sigma_{11}): H^1_*(S\otimes \Sigma_i)
\to H^1_*(S\otimes \Sigma_i)$ is an isomorphism. Lemma \ref{socle} shows that $\sigma_{11}$ is an isomorphism from $S\to S$.

Now we can apply automorphisms to $A_E, A_{E'}$ and assume that $\sigma_1 = \bmatrix 1_S & 0\\ 0 & \sigma_{22}\endbmatrix$.
The bundle $E'$ has no free summands. Hence any section $s$ of $E'$ maps to $(a,b)\in S \oplus P_{E'}$ where $b$ is a column
of polynomials where no entry in the column is a non-zero scalar. Let $e$ be a minimal generator for $H^0_*(P_{E'})$ (for
example $[1,0,\dots,0]^\vee$). From the commuting diagram of long exact sequences, we see that there is a section $s$ of $E'$
such that $(-a, e-b)$ is in the image of $\sigma_1$. Hence $e-b$ is in the image of $\sigma_{22}$, which says that
$\sigma_{22}$ is surjective on global sections.

In particular, now $\sigma: E\to E'$ is also surjective as a map of sheaves. Thus $E$ has rank not less than that of $E'$.
The discussion can be reversed to show that $E,E'$ have the same rank, and so $\sigma: E\to E'$ is really an isomorphism.
\end{proof}

We finally ask the question: For given data of invariants, is there a vector bundle $E$ on $\Q$ with those invariants?

\begin{theorem}\label{existence} (\textbf{Existence})\ Let $M$ be a graded $S(\Q)$-module of finite length. Let $\mu$ and
$\nu$ be two sequences of positive integers with finite support, $V,W$ the corresponding graded vector spaces. Assume that
$M_{\Sigma_1}$ has a graded vector subspace isomorphic to $W(1)$ consisting of $\Sigma_2$-socle elements and
$M_{\Sigma_2}$ has a graded
vector subspace isomorphic to $V(1)$ of $\Sigma_1$-socle elements. Then there exists a bundle $E$ with no ACM summands on
$\Q$ with these invariants.
\end{theorem}

\begin{proof}Let $F$ be the bundle associated to $M$, with $\Psi: 0 \to F \to L_1 \xrightarrow{\phi} L_0 \to 0$ given by a minimal
presentation of $M$. Consider the sequence
$$0\to F\otimes\Sigma_1(-1)\xrightarrow{[-v,u]^{\vee}} 2F \rightarrow F\otimes\Sigma_2\rightarrow 0$$
In the cohomology sequence $  H^0_*(F\otimes\Sigma_2) \to H^1_*(F\otimes\Sigma_1(-1)) \xrightarrow{[-v,u]^{\vee}} H^1_*(2F) $
obtained from this, the graded vector subspace $W$ is annihilated by $[-v,u]^{\vee}$, hence there is a subspace in
$H^0_*(F\otimes\Sigma_2)$ (not necessarily unique) that maps isomorphically to $W$. By abuse of notation, we call this also
as $W$. We get a map $W \otimes_k \sO \xrightarrow{\alpha_2} F\otimes\Sigma_2$.  There is an induced map $W \otimes \Sigma_2^{-2}
\xrightarrow{\alpha_2} F\otimes\Sigma_1(-1)$ which gives, on the level of $H^1_*$, the subspace $W(1)$ of $M_{\Sigma_1}$ by
Lemma \ref{socle}.

Likewise, we get a map $V \otimes_k \sO \xrightarrow{\alpha_1} F\otimes\Sigma_1$. Together, they give a map
$$ (V\otimes \Sigma_1^{-1}) \oplus  (W\otimes \Sigma_2^{-1}) \xrightarrow{(\alpha_1, \alpha_2)}  F.$$

To find a vector bundle as the quotient, we will add an $L'$ to $F$. Specifically, consider the dual map $ F^{\vee}
\xrightarrow{[\alpha_1, \alpha_2]^{\vee}}  [(V\otimes \Sigma_1^{-1}) \oplus  (W\otimes \Sigma_2^{-1})]^{\vee}$. We can find a free
$S(\Q)$-module surjecting onto the graded module $C=\text{coker} H^0_*([\alpha, \beta]^{\vee})$ and then lift the surjection
to get a map to $H^0_*([(V\otimes \Sigma_1^{-1}) \oplus  (W\otimes \Sigma_2^{-1})]^{\vee})$. If ${L'}^{\vee}$ is the free bundle
obtained by sheafifying this free module, we get a surjection of vector bundles
$$ F^{\vee}\oplus {L'}^{\vee} \xrightarrow{\bmatrix \alpha_1^{\vee} & * \\ \alpha_2^{\vee}& * \endbmatrix}
 [(V\otimes \Sigma_1^{-1}) \oplus  (W\otimes \Sigma_2^{-1})]^{\vee}$$
which is also a surjection on the level of all global  sections.

Define $E$ as the dual of the kernel of this map. We get the sequence
$$ 0 \to (V\otimes \Sigma_1^{-1}) \oplus  (W\otimes \Sigma_2^{-1}) \xrightarrow{\bmatrix \alpha_1 &  \alpha_2 \\ * & * \endbmatrix}  F \oplus L' \to E \to 0.$$

By our construction and Serre duality, the map $\bmatrix \alpha_1 &  \alpha_2 \\ * & * \endbmatrix$ induces an inclusion on
the level $H^2$, hence $H^1_*(F) \cong H^1_*(E)$.

Next, consider the  sequence $0 \to F \oplus L' \xrightarrow{\bmatrix * &  0 \\ 0 & I \endbmatrix} L_1 \oplus L' \xrightarrow
{[\phi,0]} L_0 \to 0$. We get the composite inclusion $(V\otimes \Sigma_1^{-1}) \oplus  (W\otimes \Sigma_2^{-1}) \to F \oplus
L'  \to L_1 \oplus L'$. This inclusion has a dual  map $[L_ 1 \oplus L']^{\vee} \to [(V\otimes \Sigma_1^{-1}) \oplus  (W\otimes
\Sigma_2^{-1})]^{\vee}$ which is surjective on global sections. Since a minimal surjection to $\Sigma_i$ is given in the
canonical sequences
(\ref{spinor}), it follows that the kernel can be identified with $(V^{\vee} \otimes \Sigma_1^{-1}) \oplus  (W^{\vee}\otimes
\Sigma_2^{-1}) \oplus P^{\vee}$, where $P$ is a free summand of $L_ 1 \oplus L'$.   This results in the commuting diagram

$$\begin{array}{ccccccc}
 &0 & & 0& & & \\
  &\downarrow & & \downarrow& & &\\
 & (V\otimes \Sigma_1^{-1}) \oplus  (W\otimes \Sigma_2^{-1})               & = & (V\otimes \Sigma_1^{-1}) \oplus  (W\otimes \Sigma_2^{-1})  &    &  & \\
 &\downarrow & &\downarrow   & & & \\
0 \to &F\oplus L'&\xrightarrow{} & L_1\oplus L' &\rightarrow & L_0& \to 0\\
 &\downarrow & & \downarrow & &\| &\\
0 \to  &E& \xrightarrow{} & (V \otimes \Sigma_1) \oplus  (W\otimes \Sigma_2) \oplus P&\rightarrow &L_0 &\to 0 \\
 &\downarrow & &\downarrow & & & \\

 & 0&          & 0& & & \\
 \end{array}$$
The lowest row can be compared with the sequence $0 \to E \to A_E \to L_0 \to 0$. In both sequences,
$H^0_*(L_0)$ is mapped surjectively and minimally to $H^1_*(E)$, hence by Lemma \ref{liftings}(2), the two sequences are
isomorphic.
Therefore $\mu,\nu$ coincide with $\mu_E, \nu_E$ and the  diagram above is a diagram of $E$.

\end{proof}

\end{section}

\begin{section}{Examples}

\begin{example}
\end{example}
Let  $M=k$ in degree $0$. Then, denoting $\Omega^i_{\mathbf P^3}$ as $\Omega^i$,  the vector bundle $F$ associated to $M$ is
$F=\Omega^1_{|\cal Q}$. For $i=1,2$, the auxiliary modules  $M_{\Sigma_i}$ are both equal to $k^{\oplus 2}$, concentrated
in degree $0$. Hence any element of
$M_{\Sigma_i}$ is a $\Sigma_j$-socle element ($j\neq i$). There are several choices for $V,W$, each giving a bundle $E$
without ACM summands.

\begin{enumerate}
    \item Minimally, when $V=W=0$, $E=F=\Omega^1_{|\cal Q}$.
    \item In the maximal case, when $V=k^{\oplus 2}$ in degree $-1$ and $W=k^{\oplus 2}$ in degree $-1$ ($\nu=\nu_{-1}=2$ and
    $\mu=\mu_{-1}=2$), the bundle $E$ is $\Omega^2_{|\cal Q}(2)$. The diagram for $E$ in this case requires additional line
    bundles to supplement $F$:
$$\begin{array}{ccccccc}
 &0 & &0& & & \\
  &\downarrow & &\downarrow & & &\\
 & 2\Sigma_1(-2)\oplus 2\Sigma_2(-2) &=&2\Sigma_1(-2)\oplus 2\Sigma_2(-2)    &    &  & \\
 &\downarrow & &\downarrow & & & \\
0 \to &\Omega^1_{|\cal Q}\oplus 4\sO(-1)&\rightarrow & 8\sO(-1) &\rightarrow & \sO& \to 0\\
 &\downarrow & &\downarrow & &\| &\\
 0 \to &\Omega^2_{|\cal Q}(2)& \rightarrow & 2\Sigma_2(-1)\oplus 2\Sigma_1(-1) &\rightarrow & \sO&  \to 0\\
 &\downarrow & &\downarrow & & & \\

 & 0&          &0& & & \\
 \end{array}.$$
    \item When $W=k^{\oplus 2}$ in degree $-1$ and $V=0$ ($\nu=\nu_{-1}=2$ and $\mu=0$), the vector bundle $E$ is
    $\Sigma_1^{-2}= \sO(-2,0)$, the kernel of $2\Sigma_2(-1)\xrightarrow{}\sO$.

    \item In the next three cases, the bundle depends on parameters. When $W=k$ in degree $-1$ and $V=0$
    ($\nu=\nu_{-1}=1$ and $\mu=0$), there is a one parameter family of bundles corresponding to the choice of the
    subspace $W$ of $k^{\oplus 2}$. In fact, $\Hom(\Sigma_1(-2),\Omega^1_{|\cal Q})$ is two dimensional, and each non-zero
    element gives an inclusion of bundles whose cokernel is a rank two bundle $E$ without ACM summands (giving an
    $\eta$-sequence).  An example of such a bundle $E$ is the rank two bundle given as the kernel of
    $\Sigma_2(-1)\oplus 2\sO(-1) \xrightarrow{[s,x_2,x_3]}\sO$ (giving a $\gamma$-sequence).
      \item When $W=k^{\oplus 2}$ in degree $-1$ and $V=k$  in degree $-1$ ($\nu=\nu_{-1}=2$ and
    $\mu=\mu_{-1}=1$), an example of a vector
    bundle $E$ without ACM summands corresponding to such a choice is the rank two bundle given as the kernel of
    $ 2\Sigma_2(-1)\oplus \Sigma_1(-1) \xrightarrow{[s,t,u]} \sO$. The bundles in this one-parameter family can be
    also viewed as extensions of the type $0 \to \sO(-2,0)  \to E \to \Sigma_1(-1) \to 0$.
    The diagram is
    $$\begin{array}{ccccccc}
 &0 & &0& & & \\
  &\downarrow & &\downarrow & & &\\
 & 2\Sigma_1(-2)\oplus \Sigma_2(-2) &=&2\Sigma_1(-2)\oplus \Sigma_2(-2)    &    &  & \\
 &\downarrow & &\downarrow & & & \\
0 \to &\Omega^1_{|\cal Q}\oplus 2\sO(-1)&\rightarrow & 6\sO(-1) &\rightarrow & \sO& \to 0\\
 &\downarrow & &\downarrow & &\| &\\
 0 \to &E& \rightarrow & 2\Sigma_2(-1)\oplus \Sigma_1(-1) &\rightarrow & \sO&  \to 0\\
 &\downarrow & &\downarrow & & & \\

 & 0&          &0& & & \\
 \end{array}.$$
    \item When $W=k$ and $V=k$  in degree $-1$ ($\nu=\nu_{-1}=1$ and
    $\mu=\mu_{-1}=1$), the bundles are parametrized by $\pee1 \times \pee1$. Such a bundle $E$ has rank two and
    has diagram

$$\begin{array}{ccccccc}
 &0 & &0& & & \\
  &\downarrow & &\downarrow & & &\\
 & \Sigma_1(-2)\oplus \Sigma_2(-2) &=&\Sigma_1(-2)\oplus \Sigma_2(-2)    &    &  & \\
 &\downarrow & &\downarrow & & & \\
0 \to &\Omega^1_{|\cal Q}\oplus \sO(-1)&\rightarrow & 5\sO(-1) &\rightarrow & \sO& \to 0\\
 &\downarrow & &\downarrow & &\| &\\
 0 \to &E& \rightarrow & \Sigma_2(-1)\oplus \Sigma_1(-1)\oplus \sO(-1) &\rightarrow & \sO&  \to 0\\
 &\downarrow & &\downarrow & & & \\

 & 0&          &0& & & \\
 \end{array}.$$
 $E(1)$ has a unique section whose zero-scheme is a point on $\cal Q$.

\end{enumerate}

\begin{example}
\end{example}
 Let $M=k^{\oplus 2}$ in degree $-1$, so $F= 2\Omega^1_{|\cal Q}(1)$ and  $M_{\Sigma_i}\cong k^{\oplus 4}$ in degree $-1$
 for $i=1,2$.
 Let us consider  $V=k^{\oplus 2}$ in degree $0$ and $W=k^{\oplus 2}$ in degree $0$ ($\nu=\nu_{0}=2$ and $\mu=\mu_{0}=2$).
The diagram for a rank two bundle $E$ written below that is associated to these invariants certainly exists for $E$ equal to
$[\sO(-2,0)\oplus \sO(0,-2)](1)$ as the direct sum of diagrams for  $\sO(-2,0)(1)$ in part (2) of the example above and the
corresponding one for $\sO(0,-2)(1)$:
$$\begin{array}{ccccccc}
&0 & &0& & & \\
&\downarrow & &\downarrow & & &\\
& 2\Sigma_1(-1)\oplus 2\Sigma_2(-1)               &=&2\Sigma_1(-1)\oplus 2\Sigma_2(-1)   &    &  & \\
&\downarrow & &\downarrow & & & \\
0 \to &2\Omega^1_{|\cal Q}(1)&\rightarrow & 8\sO &\rightarrow & 2\sO(1)& \to 0\\
&\downarrow & &\downarrow & &\| &\\
0 \to &E& \rightarrow & 2\Sigma_2\oplus 2\Sigma_1 &\rightarrow & 2\sO(1)&  \to 0\\
&\downarrow & &\downarrow & & & \\

& 0&          &0& & & \\
\end{array}$$

Let $N$ be a null correlation bundle on $\pee3$. Then $N$ is a stable
bundle of rank 2 on $\pee3$ with $c_1=0,$ $c_2=1$, and is obtained from a sequence
$$0\rightarrow \OP3(-1)\xrightarrow{s} \Omega^1(1) \rightarrow N\rightarrow 0.$$
Let $E=N|_{\cal Q}$. Then $H^0(E)=0$, $c_1=0, c_2=2$. Any rank two bundle $E$ on $\Q$ with the properties $H^0(E)=0, c_1=0,
c_2=2$ satisfies $M(E)= k^{\oplus 2}$ in degree $-1$, using Riemann-Roch. Then the $\gamma$-sequence $0 \to E \to A_E \to
2\sO(1) \to 0$ is easily analyzed to show that $A_E= 2\Sigma_2\oplus 2\Sigma_1$. So $E$
has the indicated diagram.

Any rank two bundle $E$ on $\Q$ with $H^0(E)=0,$ $c_1=0,$ $c_2=2$ is Mumford-Takemoto stable (\cite{Ma}). Le Potier
\cite{LP}
has a stronger definition of stability for such a bundle on $\Q$. Suppose $H^0(E)=0,$ $c_1=0,$ $c_2=2$. If
$H^0(E(a,b))\neq 0$ with $b < 0$, and if $E(a-1,b)$, $E(a,b-1)$ have no
sections, then $a\geq 1$ and a section of $H^0(E(a,b))$ determines a subscheme $Z$ with finite support (or empty) with
$\text{degree}(Z) =2+2ab \geq 0$. Hence if $H^0(E(a,b))\neq 0$ and either $b<0$ or $a<0$, then $E(a,b)=E(1,-1)$ or $E(-1,1)$.
Le Potier stability for a rank two bundle $E$ with $c_1=0,$ $c_2=2$ requires that $H^0(E)=0, H^0(E(-1,1))= 0, H^0(E(1,-1))= 0$.
This coincides with the notion of a stable projectable bundle due to Soberon-Chavez \cite{So}.

Using the invariants $V=k^{\oplus 2}\subset H^1(2\Omega^1_{|\cal Q}(1)\otimes \Sigma_2)$ and
$W=k^{\oplus 2} \subset H^1(2\Omega^1_{|\cal Q}(1)\otimes \Sigma_1)$, we can distinguish three types of bundles $E$
on $\Q$ with the above diagram: split,  non-split non-le Potier stable
and le Potier stable. For convenience, in $F= 2\Omega^1_{|\cal Q}(1) = k^{\oplus 2} \otimes_k \Omega^1_{|\cal Q}(1)$, let
$\{e_1,e_2\}$ be a basis for $k^{\oplus 2}$ and let $e_1\Omega^1_{|\cal Q}(1)$ and
$e_2\Omega^1_{|\cal Q}(1)$ denote the corresponding summands.\\
\begin{enumerate}
    \item We have already seen that if $W = H^1(e_1\Omega^1_{|\cal Q}(1)\otimes \Sigma_1)$ and $V = H^1(e_2\Omega^1_{|\cal Q}(1)
    \otimes    \Sigma_2)$,    then $E$ is split as $[\sO(-2,0)\oplus \sO(0,2)](1)$.
    \item If $W = H^1(e_1\Omega^1_{|\cal Q}(1)\otimes \Sigma_1)$ (without conditions on $V$), we get an inclusion of left columns
    of the diagrams for $\sO(-2,0)(1)$ and $E$, yielding a short sequence $$0 \to \sO(-1,1)\to E \to \sO(1,-1)\to 0.$$
    $E$ is not le Potier stable  since $H^0(E(1,-1))\neq 0$. \\
    Conversely, if $E$ is not le Potier stable, and also satisfies our diagram, it must be that
    $h^0(E(1,-1))\neq 0$ or $h^0(E(-1,1))\neq 0$. By a Chern class computation, we get, in the
    first case, an exact sequence  $0 \to \sO(-1,1)\to E \to \sO(1,-1)\to 0$. The left columns of the diagrams for
    $\sO(-2,0)(1)$ and $E$ can be completed to a commuting diagram since $\Ext^1(\Omega^1_{|\cal Q}(1), \Sigma_i(-1))=0$.
    It follows that for some basis $\{e_1,e_2\}$, $E$ is obtained with $W =
    H^1(e_1\Omega^1_{|\cal Q}(1)\otimes \Sigma_1)$.
    \item For general choices of $V,W$, the $E$ we obtain is a le Potier stable bundle on  $\Q$, giving a point in the
    moduli space $M_{\cal Q}^{lp}(0,2)$ of le Potier stable bundles.
\end{enumerate}

Le Potier \cite{LP} analyzes the restriction to the quadric $\Q$ of null correlation bundles $N$. Let $M_{\mathbf
P^3}^0(0,1)$ be the open subset of $M_{\mathbf P^3}(0,1)$ consisting of all bundles $N$ such that $N|_{\cal Q}$ is le Potier
stable on $\Q$. He shows that restriction gives an \`etale quasi-finite morphism from $M_{\mathbf P^3}^0(0,1)$ onto an
open proper subset $U \subset M_{\cal Q}^{lp}(0,2)$. The generic bundle $E$ of $U$ has a  twin pair (a Tjurin
pair) of null correlation bundles restricting to it, while there are bundles $E$ in $U$ with a unique null correlation bundle
restricting to it. Soberon-Chavez \cite{So} compactifies $M_{\cal Q}^{lp}(0,2)$ using only the non-split non-le Potier
stable bundles described above.

A complete discussion  of the diagram in this example, which we have not carried out, would describe how
the choices of $V,W$ determine whether the bundle $E$ obtained is in the subset $U$ of $M_{\cal Q}^{lp}(0,2)$, and if so,
if $E$ has a Tjurin pair over it.  We
have also not discussed strictly semi-stable bundles with $c_1=0$, $c_2=2$ on $\Q$ with $H^0(E) \neq 0$ since such bundles
$E$ have a different $M(E)$, hence a different diagram. However, we give an example below of a bundle not in $U$.

If $E$ is a bundle with the diagram above, then the $\gamma$-sequence gives rise to a homomorphism $g=(g_1,g_2):
2\Sigma_2\oplus 2\Sigma_1 \rightarrow  2\sO(1)$ and $E$ is le Potier stable if and only if both $\det g_1$, $\det g_2$
(quadratics in $s,t$, respectively $u,v$) are non-zero. It is easy to calculate that for such an $E$, the jumping lines
contained in the two pencils on $\Q$ are obtained from the zeroes of $\det g_1$ and $\det g_2$ . The restriction of a
null-correlation bundle $N$ to $\Q$ has a $\gamma$-sequence which can be calculated. When $N|_{\cal Q}$ is le Potier stable,
we can verify le Potier's observation about the jumping lines of $N|_{\cal Q}$ and see that it is not possible for both $\det
g_1$ and $\det g_2$ to have repeated roots. Now consider the map
$$g = \bmatrix s & t & u & v \\ -t & s-2t & v &0 \endbmatrix.$$
We can check that $g$ is a surjection of bundles on $\Q$, giving a rank two bundle $E$, and that
$H^0(E)=H^0(E(1,-1))=H^0(E(-1,1))=0$. Hence this $E$ is le Potier stable but not in the image $U$ of restrictions of
null-correlation bundles.

\end{section}

\end{document}